\documentclass[11pt]{article}
\usepackage{amsmath,amssymb,amsthm,mathrsfs}
\usepackage{enumitem}
\usepackage{geometry}
\usepackage{graphicx}

\numberwithin{equation}{section}

\newtheorem{theorem}{Theorem}[section]
\newtheorem{lemma}[theorem]{Lemma}

\newtheorem{remark}[theorem]{Remark}

\newcommand{\PP}{\mathbb{P}}

\newcommand{\BE}{\mathbb{E}}
\newcommand{\D}{\mathrm{d}}

\newcommand{\vertk}{\stackrel{{\cal D}}{\longrightarrow}}

\title{The limit law of the largest interpoint distance in a $d$-dimensional ellipsoid}
\author{Norbert Henze\footnote{Institute of Stochastics, Karlsruhe Institute of Technology (KIT), 76131 Karlsruhe, Germany. e-mail:
henze@kit.edu} \ and Sreenivasa Rao Jammalamadaka\footnote{Department of Statistics and Applied Probability, University of California, Santa Barbara, CA, 93106, USA. e-mail: rao@pstat.ucsb.edu}}

\date{}

\begin{document}
\maketitle

\begin{abstract}
We consider the largest interpoint distance
$M_n=\max_{1\le i<j\le n}\|X_i-X_j\|$ among independent random points $X_1,\ldots,X_n$, uniformly distributed on a $d$-dimensional ellipsoid. We assume that the largest semi-axis has length 1 and multiplicity $k\ge 2$, whereas the remaining semi-axes are strictly
smaller. In this situation, the diameter is attained on a manifold of
dimension $k-1$, and the extremal points are no longer isolated.
We establish a weak limit law for the diameter deficit $2-M_n$. Writing $q=d-k$ and
$\alpha=q+(k+3)/2$, we show that $n^{2/\alpha}(2-M_n)$ converges in distribution to a Weibull random variable. The proof is based
on a local analysis near the diameter manifold, a sharp asymptotic formula
for the two-point tail probability, and a Chen--Stein Poisson
approximation for rare nearly diametral pairs.
\end{abstract}

\medskip


\section{Introduction and statement of the main result}\label{secintro}

Let
\[
M_n
=
\max_{1 \le i < j \le n} \|X_i - X_j\|
\]
denote the largest Euclidean interpoint distance among independent and identically distributed random points $X_1,\ldots,X_n$ in $\mathbb{R}^d$, where $d \ge 2$. The asymptotic behaviour of $M_n$ has attracted considerable attention in probability theory, stochastic geometry and multivariate extreme-value theory. If the common distribution of the points has compact support $K$ then $M_n \to \operatorname{diam}(K)$ almost surely, where $\operatorname{diam}(K) = \sup \bigl\{ \|x-y\| : x,y \in K \bigr\}$. A natural problem is therefore to determine normalizing constants
$a_n \to \infty$ such that $a_n \bigl( \operatorname{diam}(K) - M_n \bigr)$ converges in distribution to a nondegenerate limit law.

Weak limit laws for the largest interpoint distance have been obtained in a variety of settings. Matthews and Rukhin \cite{mr93} derived the asymptotic distribution of the sample range for multivariate Gaussian distributions, and Henze and Klein \cite{hekl96} extended these results to symmetric Kotz-type distributions. More general results were obtained by Jammalamadaka and Janson~\cite{jamsva15} if the underlying distribution is spherically symmetric and has Gumbel-type radial tails.  For heavy-tailed spherically decomposable distributions, Henze and Lao \cite{henzelao} derived non-classical weak limit laws related to Poisson point processes. Further contributions include Appel, Najim and Russo \cite{anr} for compact planar regions, Mayer and Molchanov \cite{maymol07} for the unit ball, as well as related articles by Lao \cite{lao10},
Demichel, Fermin and Soulier \cite{dfs15}, and Heiny and Kleemann \cite{hk25}.

A particularly important line of research concerns distributions supported by ellipsoids and more general smooth convex bodies. In this context, Schrempp \cite{sc16,sc19} established weak limit laws for the largest interpoint distance in several geometric settings. In particular, \cite{sc16} treated ellipsoids with a unique major axis. In that case the diameter is attained only at two antipodal points, and the asymptotic behaviour of $M_n$ is determined by the local geometry near these isolated extremal points.

The present paper is concerned with the fundamentally different situation in
which the largest semi-axis is not unique. Recently, the first author
\cite{he26} investigated the simplest nontrivial example, namely the
rotational ellipsoid in $\mathbb R^3$ whose diameter points form an
equatorial circle. It was shown there that the correct normalization is
$a_n = n^{4/7}$, leading to a Weibull-type limit law. The proof
combined geometric localization arguments with a Chen--Stein Poisson
approximation for rare nearly diametral pairs.

The purpose of the present paper is to extend the results of \cite{he26} to arbitrary dimensions and arbitrary multiplicities of the largest semi-axis. Since the asymptotic behaviour of the largest interpoint distance is invariant under multiplication of the underlying point configuration by a positive constant, we may and shall assume without loss of generality that the largest semi-axis length equals $1$. 

Thus, let
\[
E
=
\left\{
x=(u,z)\in \mathbb{R}^k\times\mathbb{R}^q :
\|u\|^2+\sum_{j=1}^q \frac{z_j^2}{a_j^2}\le 1  \right\},
\]
where $q=d-k$, $1>a_1\ge\cdots\ge a_q>0$, and $k \ge 2$. Moreover, here and in what follows, $\|\cdot\|$ denotes the Euclidean norm, where the dimension of the underlying space will always be clear from the
context.

Hence, the first $k$ semi-axes of $E$ have length $1$, whereas the remaining semi-axes have strictly smaller lengths $a_1,\ldots,a_q$. The diameter of $E$ equals $2$, and it is attained precisely by all antipodal pairs $(e,0)$ and $(-e,0)$, where 
$e\in \mathbb{S}^{k-1}:= \{x \in \mathbb{R}^k: \|x\| = 1 \}$. 

Hence the set of diameter points is the manifold
$\mathbb{S}^{k-1}\times\{0\}$, where $0$ denotes the origin in $\mathbb{R}^q$. In particular, the extremal points are not isolated whenever $k\ge2$.

Let $X_1,X_2,\ldots$ be independent and uniformly distributed random points in $E$. The aim of this paper is to determine the asymptotic behaviour of
$2-M_n$ as $n\to\infty$.

The results below show that this asymptotic behaviour is governed
by the local geometry near the manifold of diameter points. More precisely,
the normalization exponent depends on both the dimension $k-1$ of the
diameter manifold and the number $q=d-k$ of contracting directions.

\medskip

\begin{theorem}\label{thm:main}
Let
\[
\alpha = q+\frac{k+3}{2} = d+\frac32-\frac{k}{2}.
\]

Then
\[
n^{2/\alpha}(2-M_n) \vertk Y,
\]
where $Y$ has distribution function
\[
F_Y(t) = 1-  \exp\!\left( -\frac{K_{\mathcal{A}}}{2}\,t^\alpha \right), \qquad t\ge0,
\]
and $K_{\mathcal{A}}$ is the constant appearing in Lemma~\ref{lemtwop}.
\end{theorem}

\medskip

The proof combines local geometric arguments with a Poisson
approximation for rare nearly diametral pairs. The local geometry near
the manifold of diameter points is analyzed in 
Subsections~\ref{subseclocal}--\ref{subsecloali},
whereas Subsection~\ref{subsectp} establishes a sharp asymptotic formula for the
two-point tail probability. The proof is completed in 
Subsection~\ref{subsecpoiss} by
means of the Chen--Stein method.

\section{Proof of the main result}\label{secmain}

\subsection{Local coordinates}\label{subseclocal}

Throughout, $0$ denotes the origin of the underlying Euclidean space, whose dimension will always be clear from the context. Moreover, for any integer $m$,  $\mathbb{B}^m := \{x \in \mathbb{R}^m: \|x\| \le 1\}$ denotes the unit ball in $\mathbb{R}^m$. 

If two points of $E$ have distance close to the diameter $2$, then both points must lie close to the manifold $\mathbb{S}^{k-1}\times\{0\}$, and their directions in $\mathbb{S}^{k-1}$ must be nearly antipodal. To describe this local geometry, we introduce coordinates adapted to the diameter manifold.

Put $A=\operatorname{diag}(a_1,\ldots,a_q)$. Every point of $E$, except for the set $\{0\}\times A \mathbb{B}^q$, which has Lebesgue measure zero, can be represented uniquely in the form 
\[
x=x(e,\delta,w) = \bigl(\sqrt{1-\delta}\,e,Aw\bigr),
\]
where
$e\in \mathbb{S}^{k-1}$, $0\le \delta\le 1$, $w\in\mathbb{R}^q$, and
$\|w\|^2\le \delta$.

The variable $\delta$ measures the radial defect from the sphere 
$\mathbb{S}^{k-1}\times\{0\}$, while $w$ describes the displacement in the remaining $q$ directions.

Let $X$ be uniformly distributed on $E$, and write
\[
X=x(\Xi,\Delta,W)
\]
for the corresponding random coordinates. With respect to
$\sigma_{k-1}(\D e)\, \D \delta\,\D w$,
where $\sigma_{k-1}$ denotes surface measure on $\mathbb{S}^{k-1}$, the joint density of $(\Xi,\Delta,W)$ is
\begin{equation}\label{densityf} 
f(e,\delta,w) = \frac{\det A}{2\,\operatorname{Vol}_d(E)}
(1-\delta)^{(k-2)/2},
\end{equation}
on the set
\[
\mathcal D = \left\{(e,\delta,w):
e\in \mathbb{S}^{k-1},\ 0\le \delta\le 1,\ \|w\|^2\le \delta \right\}.
\]
Indeed, the usual polar-coordinate formula in $\mathbb{R}^k$ gives
$\D u = r^{k-1}\,\D r\,\sigma_{k-1}(\D e)$, where $u=re$, and, putting
$r=\sqrt{1-\delta}$,
\[
r^{k-1}\, \D r = -\frac{1}{2} (1-\delta)^{(k-2)/2}\, \D \delta .
\]
The change of variables $z=Aw$ contributes the factor $\det A$. This proves the asserted density.

Note that $f$ does not depend on $e$, reflecting the rotational symmetry
in the first $k$ coordinates, nor on $w$, since the transformation
$z=Aw$ contributes only the constant Jacobian factor $\det A$.
In particular, near the diameter manifold, that is, for $\delta \downarrow 0$, the density is asymptotically constant. This is the analogue of the constant-density property used in the three-dimensional rotational case.

\subsection{Local expansion}

We now investigate the geometry of pairs of points whose distance is close
to the diameter $2$. To this end, let  $x=x(e,\delta,w)$ and $y=x(e',\delta',w')$,
where 
\[
e,e'\in \mathbb{S}^{k-1}, \qquad 0\le\delta,\delta'\le1, \qquad
\|w\|^2\le\delta, \qquad \|w'\|^2\le\delta'.
\]
If $\|x-y\|$ is close to $2$, then $e'$ must be close to $-e$.
We therefore write $e'=-g$, where $g\in \mathbb{S}^{k-1}$ is close to $e$.
Let $\theta\in[0,\pi]$ denote the angle between $e$ and $g$. Thus
$\langle e,g\rangle=\cos\theta$.

The following result gives the local expansion of the distance deficit.

\begin{lemma}[Local expansion]\label{lemlocalex}
Let $x=x(e,\delta,w)$ and $y=x(-g,\delta',w')$, where
$e,g\in \mathbb{S}^{k-1}$ and $\langle e,g\rangle=\cos\theta$.
Then, as $\delta,\delta',\theta,\|w\|,\|w'\|\to 0$, we have
\[
 2-\|x-y\| = \frac{\delta+\delta'}{2} + \frac{\theta^2}{4}
- \frac{1}{4} \|A(w-w')\|^2 + r(\delta,\delta',\theta,w,w'),
\]
where
\[
r(\delta,\delta',\theta,w,w') = o\!\left(\delta+\delta'+\theta^2+\|w\|^2+\|w'\|^2
\right).
\]
The remainder is uniform in $e,g\in \mathbb{S}^{k-1}$.
\end{lemma}

\smallskip
\begin{proof}
Write $r_1=\sqrt{1-\delta}, r_2=\sqrt{1-\delta'}$. 
Since $x=(r_1e,Aw)$ and $y=(-r_2g,Aw')$, we obtain
\[
\|x-y\|^2 = \|r_1e+r_2g\|^2 + \|A(w-w')\|^2.
\]
Furthermore,
\[
\|r_1e+r_2g\|^2 = r_1^2+r_2^2+2r_1r_2\langle e,g\rangle.
\]
Since $\langle e,g\rangle=\cos\theta$,
and $\cos\theta = 1- \theta^2/2 +o(\theta^2)$ as $\theta\downarrow 0$,
we obtain $2r_1r_2\langle e,g\rangle = 2r_1r_2-r_1r_2\theta^2+o(\theta^2)$.
Moreover,
\[
r_1 = 1-\frac{\delta}{2}+o(\delta), \qquad
r_2 = 1-\frac{\delta'}{2}+o(\delta').
\]
Hence $r_1^2+r_2^2 = 2-\delta-\delta'$ and $2r_1r_2 = 2-\delta-\delta'+o(\delta+\delta')$. Consequently,
\[
\|x-y\|^2 = 4-2(\delta+\delta')-\theta^2 +\|A(w-w')\|^2
+ o\!\left( \delta+\delta'+\theta^2+\|w\|^2+\|w'\|^2 \right).
\]
\noindent 
Now write $\|x-y\| = 2-\varepsilon$.
Since $\varepsilon\to 0$, we have
$(2-\varepsilon)^2 = 4-4\varepsilon+o(\varepsilon)$.
Therefore
\[
4\varepsilon = 2(\delta+\delta') +\theta^2 -\|A(w-w')\|^2
+ o\!\left( \delta+\delta'+\theta^2+\|w\|^2+\|w'\|^2 \right).
\]
This proves the assertion.
\end{proof}

\subsection{Localization}\label{subsecloali}
The local expansion from the previous subsection is useful only if nearly
diametral pairs necessarily lie close to the diameter manifold
$\mathbb{S}^{k-1}\times\{0\}$. The next lemma shows that this is indeed the case.

\begin{lemma}[Localization]\label{lemlocal}
There exist constants $C>0$ and $\varepsilon_0>0$ such that the following
holds. Let $x=x(e,\delta,w)$, $y=x(e',\delta',w')$, and suppose that
$2-\|x-y\|\le\varepsilon$ for some $0<\varepsilon<\varepsilon_0$. Then
\[
\delta+\delta'\le C\varepsilon,
\]
and
\[
\|w\|^2+\|w'\|^2\le C\varepsilon.
\]
Moreover, $\langle e,e'\rangle+1\le C\varepsilon$.

Equivalently, if $\theta\in[0,\pi]$ is defined by
$\langle e,-e'\rangle=\cos\theta$,
then $\theta^2\le C\varepsilon$.
\end{lemma}

\smallskip
\begin{proof}
We first derive bounds for $\delta+\delta'$ and $\|w\|^2+\|w'\|^2$. The estimate for the angular deviation from antipodality will then follow. Write $r_1=\sqrt{1-\delta}$ and  $r_2=\sqrt{1-\delta'}$.

Then $x=(r_1e,Aw)$ and $y=(r_2e',Aw')$. Consequently,
\[
\|x-y\|^2 = r_1^2+r_2^2 - 2r_1r_2\langle e,e'\rangle + \|A(w-w')\|^2 .
\]
Since $\|w\|^2\le \delta$ and $\|w'\|^2\le \delta'$, we have
\[
\|A(w-w')\|^2 \le  a_1^2(\|w\|+\|w'\|)^2 \le 2a_1^2(\delta+\delta').
\]
Moreover, since $\langle e,e'\rangle\ge -1$, it follows that
$\|x-y\|^2 \le (r_1+r_2)^2+2a_1^2(\delta+\delta')$.
Using \(2\sqrt{uv}\le u+v\), with \(u=1-\delta\) and \(v=1-\delta'\), we obtain
\[
2r_1r_2 = 2\sqrt{(1-\delta)(1-\delta')} \le 2-\delta-\delta'.
\]
Therefore,
$(r_1+r_2)^2 = r_1^2+r_2^2+2r_1r_2 \le 4-2(\delta+\delta')$.
Consequently,
\[
\|x-y\|^2 \le 4 -2(1-a_1^2)(\delta+\delta').
\]
Put $\gamma=2(1-a_1^2)>0$. Then $\|x-y\|^2 \le 4-\gamma(\delta+\delta')$.

On the other hand, \(2-\|x-y\|\le\varepsilon\) implies
$\|x-y\|^2 \ge (2-\varepsilon)^2 = 4-4\varepsilon+\varepsilon^2$.
Combining the last two inequalities gives
\[
\delta+\delta' \le \frac{4\varepsilon-\varepsilon^2}{\gamma}
\le C\varepsilon
\]
for all sufficiently small $\varepsilon$. Here and throughout the remainder of the proof,  $C$ denotes a positive constant whose value may change from line to line. 

Since $\|w\|^2\le\delta$ and
$\|w'\|^2\le\delta'$, it follows that $\|w\|^2+\|w'\|^2\le C\varepsilon$.

It remains to estimate the angular deviation from antipodality. Since
$\delta+\delta'=O(\varepsilon)$, we have
$r_1^2+r_2^2=2+O(\varepsilon)$ and $r_1r_2=1+O(\varepsilon)$.
Moreover, $\|A(w-w')\|^2=O(\varepsilon)$.

Using the identity
\[
\|x-y\|^2 = r_1^2+r_2^2 - 2r_1r_2\langle e,e'\rangle + \|A(w-w')\|^2
\]
and the lower bound $\|x-y\|^2\ge4-4\varepsilon$, we obtain
$-2r_1r_2\langle e,e'\rangle \ge 2-C\varepsilon$.
Since $r_1r_2=1+O(\varepsilon)$, this yields
$-\langle e,e'\rangle \ge 1-C\varepsilon$ 
and hence $\langle e,e'\rangle+1\le C\varepsilon$.

Finally, let $\theta\in[0,\pi]$ be defined by
$\langle e,-e'\rangle=\cos\theta$. Then
$\langle e,e'\rangle=-\cos\theta$, and therefore
\[
1-\cos\theta = \langle e,e'\rangle+1 \le C\varepsilon.
\]
Since $1-\cos\theta\ge 2\theta^2/\pi^2$ for $0\le\theta\le\pi$, we conclude that
$\theta^2\le C\varepsilon$.
This proves the assertion.
\end{proof}


\medskip
The localization result shows that, on the event
$\{2-\|X_1-X_2\|\le\varepsilon\}$,
both points must lie within an $O(\varepsilon)$-neighbourhood of the
diameter manifold $\mathbb{S}^{k-1}\times\{0\}$, and their directions on
$\mathbb{S}^{k-1}$ must be nearly antipodal. Consequently, the asymptotic
behaviour of $\PP(2-\|X_1-X_2\|\le\varepsilon)$
is determined entirely by the local geometry described in
Lemma~\ref{lemlocalex}. The next subsection exploits this observation to derive a
sharp asymptotic formula for the two-point tail probability.

\subsection{The two-point tail}\label{subsectp}

For independent random points $X_1,X_2$, uniformly distributed on $E$, put
\[
W_{12}=2-\|X_1-X_2\|.
\]
In this subsection we determine the asymptotic behaviour of
$\PP(W_{12}\le \varepsilon)$ as $\varepsilon \downarrow 0$.
To this end, let $c_E= \det A/(2\operatorname{Vol}_d(E))$ 
denote the constant appearing in the density $f$ of the random coordinates in
\eqref{densityf}. Put
\[
\alpha = q+\frac{k+3}{2} = d+\frac32-\frac{k}{2}. 
\]
Furthermore, define
\[
G(s,s',y,y',\tau) = \frac{s+s'}{2} + \frac{\tau^2}{4} -  \frac14\|A(y-y')\|^2,
\]
where $s,s'\ge 0$, $y,y'\in\mathbb{R}^q$, and $\tau\ge 0$.

Finally, let
\[
I_\mathcal{A} = \sigma_{k-2}(\mathbb{S}^{k-2}) \int_{\mathcal{A}}
{\bf 1}\{G(s,s',y,y',\tau)\le1\} \tau^{k-2}\,\D s\, \D s'\,\D y\,\D y'\,\D\tau,
\]
where
\[
\mathcal{A} = \left\{(s,s',y,y',\tau): s,s'\ge0,\, \|y\|^2\le s,\,
\|y'\|^2\le s',\, \tau\ge0 \right\}.
\]

\medskip

\begin{lemma}[Sharp two-point tail]\label{lemtwop}
As $\varepsilon\downarrow 0$, we have
\[
\PP(W_{12}\le\varepsilon) \sim K_{\mathcal{A}}\varepsilon^\alpha,
\]
where
\begin{equation}\label{defka}
K_{\mathcal{A}} = c_E^2\,\sigma_{k-1}(\mathbb{S}^{k-1})\,I_{\mathcal{A}}.
\end{equation}
\end{lemma}

\smallskip
\begin{proof}
Let $X_1=x(\Xi,\Delta,W)$ and $X_2=x(\Xi',\Delta',W')$ be independent copies in the coordinates introduced in Subsection~\ref{subseclocal}. We
write $\Xi'=-G'$, where $G'\in \mathbb{S}^{k-1}$. Since surface measure on
$\mathbb{S}^{k-1}$ is invariant under the map $e'\mapsto -e'$, the random
variable $G'$ is again integrated with respect to surface measure on
$\mathbb{S}^{k-1}$.

For fixed $e\in \mathbb{S}^{k-1}$, write
$g=\cos\theta\,e+\sin\theta\,v$,
where $0\le\theta\le\pi$ and $v\in\mathbb{S}^{k-2}$ denotes the angular
direction orthogonal to $e$. Thus
$\langle e,g\rangle=\cos\theta$.

The standard polar-coordinate formula on $\mathbb{S}^{k-1}$ yields
\[
\sigma_{k-1}(\D g)
=
(\sin\theta)^{k-2}\,\D\theta\,\sigma_{k-2}(\D v).
\]
The limiting integrand below will not depend on $v$.

By the density formula \eqref{densityf},
\[
\PP(W_{12}\le\varepsilon)
= c_E^2 \int (1-\delta)^{(k-2)/2}(1-\delta')^{(k-2)/2}
{\bf 1}\{2-\|x-y\|\le\varepsilon\} \,\D \mu,
\]
where $x=x(e,\delta,w)$, $y=x(-g,\delta',w')$, and 
$\D \mu  = \sigma_{k-1}(\D e)\, \sigma_{k-1}(\D g)\, \D \delta\,
\D \delta'\,\D w\,\D w'$,
with $0\le\delta,\delta'\le 1$, $\|w\|^2\le\delta$, $\|w'\|^2\le\delta'$.

By Lemma~\ref{lemlocal}, the event $\{2-\|x-y\|\le\varepsilon\}$
implies, for sufficiently small $\varepsilon$,
\[
\delta+\delta'=O(\varepsilon),
\qquad
\|w\|^2+\|w'\|^2=O(\varepsilon),
\qquad
\theta^2=O(\varepsilon).
\]
We may therefore introduce the scaled variables
\[
\delta=\varepsilon s,\qquad \delta'=\varepsilon s', \qquad
w=\sqrt{\varepsilon}\,y,
\qquad
w'=\sqrt{\varepsilon}\,y',
\qquad \text{and} \qquad 
\theta=\sqrt{\varepsilon}\,\tau.
\]
The constraints $\|w\|^2\le\delta$ and $\|w'\|^2\le\delta'$ become
$\|y\|^2\le s$ and $\|y'\|^2\le s'$, and the Jacobian contribution from these changes of variables is $\varepsilon^2\varepsilon^q\varepsilon^{1/2} = \varepsilon^{q+5/2}$.
In addition,
\[
(\sin\theta)^{k-2} = (\sin(\sqrt{\varepsilon}\tau))^{k-2} \sim
\varepsilon^{(k-2)/2}\tau^{k-2},
\]
uniformly for $\tau$ in bounded intervals. Hence the full scaling factor is
\[
\varepsilon^{q+5/2}\varepsilon^{(k-2)/2} = \varepsilon^{q+(k+3)/2}
= \varepsilon^\alpha.
\]

By Lemma~\ref{lemlocalex}, after this change of variables,
\[
2-\|x-y\| = \varepsilon G(s,s',y,y',\tau) + r_\varepsilon,
\]
where
\[
r_\varepsilon = o\!\left( \delta+\delta' +\theta^2 +\|w\|^2 +\|w'\|^2 \right).
\]
On bounded subsets of the scaled parameter space we have
\[
\delta+\delta' +\theta^2 +\|w\|^2 +\|w'\|^2 = O(\varepsilon).
\]
Hence $r_\varepsilon=o(\varepsilon)$, uniformly on such bounded subsets. Consequently,
\[
\frac{2-\|x-y\|}{\varepsilon} \to G(s,s',y,y',\tau)
\]
uniformly on bounded subsets of the scaled parameter space.
Moreover,
\[
(1-\varepsilon s)^{(k-2)/2} (1-\varepsilon s')^{(k-2)/2} \to 1.
\]
Since $G$ is affine in $s$ and $s'$, the level set
$G=1$ has Lebesgue measure zero.

By Lemma~\ref{lemlocal}, on the event $\{2-\|x-y\|\le \varepsilon\}$,
we have
\[
\delta+\delta'=O(\varepsilon),
\qquad
\|w\|^2+\|w'\|^2=O(\varepsilon),
\qquad
\theta^2=O(\varepsilon).
\]
Hence, after the above change of variables,
\[
s+s'=O(1), \qquad \|y\|^2+\|y'\|^2=O(1), \qquad \tau^2=O(1).
\]
Consequently, the transformed indicators vanish outside a bounded subset
of the scaled parameter space that does not depend on sufficiently small
\(\varepsilon\).

Therefore the dominated convergence theorem applies. The
possible discontinuity of the indicator at
$G(s,s',y,y',\tau) = 1$ does not affect the limit, since this level set has Lebesgue measure zero. Consequently,
\[
\varepsilon^{-\alpha} \PP(W_{12}\le\varepsilon)
\to 
c_E^2\sigma_{k-1}(\mathbb{S}^{k-1})\sigma_{k-2}(\mathbb{S}^{k-2})
\int_{\mathcal{A}}
{\bf 1}\{G(s,s',y,y',\tau)\le1\}
\tau^{k-2}\,\D s\,\D s'\,\D y\,\D y'\,\D \tau.
\]
By the definition of $I_{\mathcal{A}}$, the right-hand side is
\[
c_E^2\sigma_{k-1}(\mathbb{S}^{k-1})I_{\mathcal{A}}
=
K_\mathcal{A}.
\]
This proves the assertion.
\end{proof}

\medskip
\begin{remark}
The constant $K_{\mathcal{A}}$ figuring in \eqref{defka} is finite and strictly positive. Indeed, it is enough to show that $0<I_{\mathcal{A}}<\infty$. Positivity follows since the set
\[
\left\{ (s,s',y,y',\tau): 0<s<\frac14,\quad 0<s'<\frac14,\quad
\|y\|^2<\frac{s}{4},\quad \|y'\|^2<\frac{s'}{4},\quad 0<\tau<1 \right\}
\]
has positive Lebesgue measure and is contained in $\mathcal A$. Moreover,
on this set,
\[
G(s,s',y,y',\tau) \le \frac{s+s'}{2}+\frac{\tau^2}{4} < \frac14+\frac14 < 1.
\]
Thus $I_{\mathcal{A}}>0$. To prove finiteness, note first that on $\mathcal A$ we have
\[
\|y-y'\|^2 \le 2\|y\|^2+2\|y'\|^2 \le 2s+2s'.
\]
Hence
\[
\|A(y-y')\|^2 \le a_1^2\|y-y'\|^2 \le 2a_1^2(s+s').
\]
Therefore
\[
G(s,s',y,y',\tau) \ge \frac{s+s'}{2} +\frac{\tau^2}{4} -\frac{a_1^2}{2}(s+s')
= \frac{1-a_1^2}{2}(s+s') +\frac{\tau^2}{4}.
\]
Consequently, $G(s,s',y,y',\tau)\le 1$ implies
\[
\frac{1-a_1^2}{2}(s+s') +\frac{\tau^2}{4} \le 1.
\]
Hence $s$, $s'$, $y$, $y'$ and $\tau$ are all confined to a bounded set.
Since the integrand $\tau^{k-2}$ is continuous for $k\ge 2$,
it follows that $I_{\mathcal{A}}<\infty$.
\end{remark}


\subsection{Poisson approximation}\label{subsecpoiss}

We now turn the two-point tail estimate into a limit theorem for the
maximum interpoint distance. For $1\le i<j\le n$, put $W_{ij}=2-\|X_i-X_j\|$.
We first need a bound for two rare events that share one common index.

\begin{lemma}\label{lem-shared-index}
As $\varepsilon\downarrow 0$,
\[
\PP(W_{12}\le\varepsilon,\ W_{13}\le\varepsilon) =
O(\varepsilon^\beta),
\]
where
\[
\beta  = k+\frac{3q}{2}+2.
\]
\end{lemma}

\begin{proof}
Let $X_i=x(\Xi_i,\Delta_i,W_i)$,  $i=1,2,3$,
where $(\Xi_1,\Delta_1,W_1)$, $(\Xi_2,\Delta_2,W_2)$ and 
$(\Xi_3,\Delta_3,W_3)$ are independent copies of the random vector
$(\Xi,\Delta,W)$ introduced in Subsection~\ref{subseclocal}.

On the event $\{W_{12}\le\varepsilon,\;W_{13}\le\varepsilon\}$,
Lemma~\ref{lemlocal}, applied to the pairs $(X_1,X_2)$ and $(X_1,X_3)$ yields 
$\Delta_1+\Delta_2\le C\varepsilon$, $\Delta_1+\Delta_3\le C\varepsilon$.
Hence $\Delta_i\le C\varepsilon$ and $\|W_i\|^2\le C\varepsilon$ for $i=1,2,3$.

Moreover, the directions $\Xi_2$ and $\Xi_3$ must both lie within angular
distance $O(\sqrt{\varepsilon})$ of $-\Xi_1$. Equivalently, each of
$\Xi_2,\Xi_3$ lies in a spherical cap on $\mathbb{S}^{k-1}$ of angular radius
$C\sqrt{\varepsilon}$.

Let $(e_i,\delta_i,w_i)$, $i=1,2,3$, be realizations of the random coordinate vectors
$(\Xi_i,\Delta_i,W_i)$, $i=1,2,3$.

Let $B_\varepsilon$ be the set of all realizations
$(e_i,\delta_i,w_i)$, $i=1,2,3$, satisfying
\[
0\le \delta_i\le C\varepsilon,\qquad
\|w_i\|^2\le C\varepsilon,\qquad i=1,2,3,
\]
and such that both $e_2$ and $e_3$ lie in the spherical cap on
$\mathbb S^{k-1}$ centred at $-e_1$ with angular radius
$C\sqrt{\varepsilon}$. By the preceding localization argument,
$\{W_{12}\le\varepsilon,\ W_{13}\le\varepsilon\}$
is contained in the event that the corresponding coordinate vectors
belong to $B_\varepsilon$.

Since the density $f$ in \eqref{densityf} is bounded, and since the three random
coordinate vectors are independent, the probability under consideration
is bounded by a constant multiple of the product-coordinate volume of
$B_\varepsilon$. It therefore remains to estimate the product-coordinate volume of
$B_\varepsilon$.

The variable $e_1$ contributes
only the finite surface measure of $\mathbb{S}^{k-1}$. The variables
$\delta_1$, $\delta_2$ and $\delta_3$ contribute a factor $\varepsilon^3$. The
variables $w_1,w_2,w_3 \in\mathbb{R}^q$ contribute
$\varepsilon^{3q/2}$. Finally, the two spherical caps for $e_2$ and $e_3$ each have surface measure of order
\[
(\sqrt{\varepsilon})^{k-1} = \varepsilon^{(k-1)/2}.
\]
Consequently,
\[
\PP(W_{12}\le\varepsilon,\ W_{13}\le\varepsilon) = 
O\!\left(\varepsilon^3 \varepsilon^{3q/2} \varepsilon^{k-1} \right)
= O(\varepsilon^\beta),
\]
where
\[
\beta = 3+\frac{3q}{2}+k-1 = k+\frac{3q}{2}+2.
\]
This proves the assertion.
\end{proof}

\medskip

\noindent For $t\ge 0$, put $\varepsilon_n=t n^{-2/\alpha}$, and define
$I_{ij}=\mathbf{1}\{W_{ij}\le\varepsilon_n\}$ for $1\le i<j\le n$. Moreover,
let
\[
N_n(t) = \sum_{1\le i<j\le n} I_{ij}
\]
be the number of pairs whose distance deficit is at most $\varepsilon_n$. Let 
$\operatorname{Po}\left(\lambda \right)$ denote the Poisson distribution with expectation $\lambda$.

\begin{lemma}\label{lem-poisson}
For each $t\ge 0$,
\[
N_n(t) \vertk  \operatorname{Po}\left(\frac{K_{\mathcal{A}}}{2}t^\alpha\right).
\]
\end{lemma}

\smallskip

\begin{proof}
By Lemma~\ref{lemtwop}, 
\[
\PP(W_{12}\le\varepsilon_n)
\sim
K_{\mathcal{A}}\varepsilon_n^\alpha =
K_{\mathcal{A}} t^\alpha n^{-2} \quad \text{as} \ n \to \infty.
\]
Hence
\[
\BE N_n(t)
=
{n\choose2} \PP(W_{12}\le\varepsilon_n) \to 
\frac{K_{\mathcal{A}}}{2}t^\alpha.
\]

We now apply Theorem~1 of Arratia, Goldstein and Gordon~\cite{agg}. To this end, 
let
\[
\mathcal I_n=\{(i,j):1\le i<j\le n\}.
\]
For $\alpha_0=(i,j)\in\mathcal I_n$, define
$B_{\alpha_0} = \{(r,s)\in\mathcal I_n:\{r,s\}\cap\{i,j\}\ne\emptyset\}$.
Then $I_{\alpha_0}$ is independent of the family
$\{I_{\beta_0}:\beta_0\notin B_{\alpha_0}\}$.

Consequently, in the notation of \cite{agg}, the
term $b_3$ is equal to zero.

Since each pair $(i,j)$ has at most $2(n-2)+1=O(n)$ neighbours in the
dependency graph, the total number of summands in the quantities $b_1$
and $b_2$ in~\cite{agg} is of order $n^3$. Therefore
\[
b_1 = \sum_{\alpha_0\in\mathcal I_n} \sum_{\beta_0\in B_{\alpha_0}}
\PP(I_{\alpha_0}=1)\PP(I_{\beta_0}=1) 
= O\!\left(n^3\PP(W_{12}\le\varepsilon_n)^2 \right).
\]
By Lemma~\ref{lemtwop}, 
\[
n^3 \PP(W_{12}\le\varepsilon_n)^2 = O(n^3\varepsilon_n^{2\alpha})
= O(n^3n^{-4}) = O(n^{-1}) \to 0.
\]
Thus $b_1\to 0$. Similarly,
\[
b_2 = \sum_{\alpha_0 \in \mathcal{I}_n}
\sum_{\substack{\beta_0\in B_{\alpha_0}\\ \beta_0\ne\alpha_0}}
\PP(I_{\alpha_0}=1,I_{\beta_0}=1).
\]
Every nonzero summand in \(b_2\) has the same form as
$\PP(W_{12}\le\varepsilon_n,\ W_{13}\le\varepsilon_n)$,
up to a relabelling of indices. Hence, by Lemma~\ref{lem-shared-index},
\[
b_2 = O\!\left( n^3 \PP(W_{12}\le \varepsilon_n,\ W_{13}\le\varepsilon_n)
\right)
= O(n^3\varepsilon_n^\beta).
\]
Since $\varepsilon_n=tn^{-2/\alpha}$, this gives
\[
b_2 = O\!\left(n^{3-2\beta/\alpha} \right).
\]
Moreover,
\[
\beta-\frac32\alpha = \left(k+\frac{3q}{2}+2\right)
- \frac32\left(q+\frac{k+3}{2}\right) = \frac{k-1}{4}.
\]
Since $k\ge 2$, we have $\beta > 3\alpha/2$ 
and hence
\[
3-\frac{2\beta}{\alpha}<0.
\]
Therefore $b_2\to 0$.
Theorem~1 of~\cite{agg} now yields
\[
\lim_{n\to \infty} d_{\rm TV}\left( \mathcal L(N_n(t)), \operatorname{Po}(\BE N_n(t)) \right)
 = 0.
\]
Together with $\BE N_n(t)\to K_{\mathcal{A}}t^\alpha /2$,
this proves the assertion. 
\end{proof}

\bigskip
\noindent 
We can now complete the proof of the main theorem.

\begin{proof}[Proof of Theorem~\ref{thm:main}]
Let $t\ge 0$, and put $\varepsilon_n=t n^{-2/\alpha}$.
By the definition of $N_n(t)$,
\[
\{N_n(t)=0\}
=
\{W_{ij}>\varepsilon_n \text{ for all } 1\le i<j\le n\}.
\]
Since $W_{ij}=2-\|X_i-X_j\|$, this event is the same as
$\{2-M_n>\varepsilon_n\}$.

Equivalently,
\[
\{N_n(t)=0\}
=
\{n^{2/\alpha}(2-M_n)>t\}.
\]

By Lemma~\ref{lem-poisson},
\[
N_n(t)
\stackrel{\mathcal D}{\longrightarrow}
\operatorname{Po}\!\left(\frac{K_{\mathcal{A}}}{2}t^\alpha\right).
\]
Therefore
\[
\PP\!\left(n^{2/\alpha}(2-M_n)>t\right)
= \PP(N_n(t)=0)
\to
\exp\!\left(-\frac{K_{\mathcal{A}}}{2}t^\alpha\right).
\]
Hence
\[
\mathbb{P}\!\left(n^{2/\alpha}(2-M_n)\le t\right)
\to 1-\exp\!\left(-\frac{K_{\mathcal{A}}}{2}t^\alpha\right), \qquad t\ge 0.
\]
This proves the assertion.
\end{proof}

\section{Remarks and open problems}

\begin{remark}
  
Theorem~\ref{thm:main} extends the recently obtained limit law for the
largest interpoint distance in a rotational ellipsoid in $\mathbb R^3$
studied in \cite{he26}. Indeed, for $d=3$, $k=2$, and $q=1$, we obtain
\[
\alpha = q+\frac{k+3}{2} = \frac72.
\]
Hence the normalization exponent is $2/\alpha=4/7$, which recovers the
$4/7$-limit law established in \cite{he26}.
\end{remark}

\medskip
 \begin{remark}
Although Theorem~\ref{thm:main} is proved only under the assumption
$k\ge 2$, the exponent appearing in Theorem~\ref{thm:main} is consistent with the case $k=1$ treated by
Schrempp \cite{sc16}. Indeed, if one formally puts $k=1$, then
$q=d-1$, and hence
\[
\alpha = q+\frac{k+3}{2} = (d-1)+2 = d+1.
\]
Consequently, the normalization exponent becomes
\[
\frac{2}{\alpha} = \frac{2}{d+1},
\]
which is precisely the exponent obtained in \cite{sc16} for ellipsoids
with a unique major axis. It is worth noting that Schrempp's proof is
based on a substantially different approach, exploiting the local geometry
near two isolated diameter points. By contrast, the present paper deals
with the fundamentally different situation in which the diameter is
attained on a manifold of positive dimension and combines a local analysis
near this manifold with a Chen--Stein Poisson approximation.
\end{remark}

\medskip

A first natural problem is to replace the uniform distribution on the
ellipsoid $E$ by a more general absolutely continuous distribution.
One may conjecture that the normalization exponent
$\alpha=q+(k+3)/2$  remains unchanged whenever the density is continuous and strictly
positive in a neighbourhood of the diameter manifold
$\mathbb{S}^{k-1}\times\{0\}$.
In such a setting, however, the constant appearing in the limit law
should depend not only on the local geometry of the support but also on
the values of the density along the diameter manifold.

A second problem is to extend the analysis to bounded regions that are
more general than the ellipsoids considered in the present paper. The
explicit geometry of the ellipsoid makes it possible to introduce
suitable local coordinates and to derive a tractable expansion for the
distance deficit. It would be interesting to identify geometric
conditions under which analogous arguments remain valid for more general
convex bodies.

More generally, one may ask for weak limit laws when the diameter is
attained along a smooth manifold of positive dimension. The results of
the present paper suggest that the normalization exponent should depend
on both the dimension of the diameter manifold and the local behaviour
of the boundary in directions transversal to that manifold. Determining
the precise form of this dependence appears to be an open problem.

It would  be desirable to develop a general theory for compact
supports whose diameter is attained on a manifold of extremal pairs. The
present paper and \cite{he26} indicate that the asymptotic behaviour of
the largest interpoint distance is governed by a delicate interplay
between local geometry and the amount of probability mass concentrated
near the diameter set.

Finally, it  would also be interesting to investigate possible applications of the present results in directional statistics and in the statistical analysis of high-dimensional point clouds arising, for example, in astronomy.

\end{document}